\documentclass[letterpaper]{amsart}
\usepackage{amsfonts,amssymb,amscd,amsmath,enumerate,url,verbatim}
\usepackage[all]{xy}
\SelectTips{eu}{}
\xyoption{curve}
\CompileMatrices

\newcommand{\ov}{\overline}

\newcommand{\ges}{{\scriptscriptstyle\geqslant}}
\newcommand{\col}{\colon}
\newcommand{\dd}{\partial}
\newcommand{\fm}{{\mathfrak m}}
\newcommand{\ann}{\operatorname{ann}}

\newcommand{\fn}{{\mathfrak n}}
\newcommand{\fp}{{\mathfrak p}}
\newcommand{\fq}{{\mathfrak q}}
\newcommand{\bw}{{\mathsf\Lambda}}

\newcommand{\cls}{\operatorname{cls}}
\newcommand{\Soc}{\operatorname{Soc}}
\newcommand{\Ima}{\operatorname{Im}}

\newcommand{\Ker}{\operatorname{Ker}}

\newcommand{\edim}{\operatorname{edim}}
\newcommand{\reg}{\operatorname{reg}}

\newcommand{\rank}{\operatorname{rank}}
\newcommand{\HH}{\operatorname{H}}

\newcommand{\Tor}{\operatorname{Tor}}
\newcommand{\Ext}{\operatorname{Ext}}

\newcommand{\Hom}{\operatorname{Hom}}
\newcommand{\Po}{\operatorname{P}}

\newcommand{\agr}[2][{}]{{{#2}^{\mathsf g}_{#1}}}

\newcommand{\bsh}{{\boldsymbol h}}

\theoremstyle{plain}
\newtheorem{theorem}{Theorem}[section]

\newtheorem*{Main Theorem}{Main Theorem}
\newtheorem*{Corollary}{Corollary}

\newtheorem{proposition}[theorem]{Proposition}
\newtheorem{lemma}[theorem]{Lemma}
\newtheorem{corollary}[theorem]{Corollary}

{\Alph{theorem}}
{\Alph{theorem}}

\theoremstyle{definition}
\newtheorem*{definition}{Definition}

\newtheorem{chunk}[theorem]{}

\newtheorem{remark}[theorem]{Remark}

\theoremstyle{remark}

\newenvironment{bfchunk}{\begin{chunk}\textit}{\end{chunk}}

\numberwithin{equation}{theorem}

\begin{document}
\title[Compressed Gorenstein rings]{Poincar\'e series of modules\\ over compressed Gorenstein local rings}

\begin{abstract}
Given two positive integers $e$ and $s$ we consider  Gorenstein Artinian local  rings $R$ whose maximal ideal $\fm$  satisfies $\fm^s\ne 0=\fm^{s+1}$ and $\rank_{R/\fm}(\fm/\fm^2)=e$.  We say that $R$  is a {\it compressed  Gorenstein local ring}  when it  has  maximal length among such rings.  It is known that generic Gorenstein Artinian algebras are compressed. If $s\ne 3$, we prove that the Poincar\'e series of all finitely generated modules over a compressed Gorenstein local ring are rational, sharing a common denominator. A formula for the denominator is given. When $s$ is even this formula depends only on the integers $e$ and $s$.  Note that for  $s=3$ examples of compressed Gorenstein local rings with transcendental Poincar\'e series exist, due to B\o gvad. 
\end{abstract}

\author[M.~E.~Rossi]{Maria Evelina Rossi}
\address{Maria Evelina Rossi\\ Department of Mathematics of  Genoa \\ Via Dodecaneso 35, 16146-Genova, Italy}
\email{rossim@dima.unige.it}

\author[L.~M.~\c{S}ega]{Liana M.~\c{S}ega}
\address{Liana M.~\c{S}ega\\ Department of Mathematics and Statistics\\
   University of Missouri\\ \linebreak Kansas City\\ MO 64110\\ U.S.A.}
     \email{segal@umkc.edu}

\thanks{MER was  supported by PRA grant 2011  and LM\c S by NSF grant DMS-1101131 and Simons Foundation grant 20903.  This work was also supported by the NSF grant No. 0932078 000, while the authors were in residence at the Mathematical Science Research Institute (MSRI) in Berkeley, California in Fall 2012 (MER) and Spring 2013 (LM\c S). }

\maketitle
\section*{Introduction}

Let $(R,\fm,k)$ be a local commutative ring with maximal ideal $\fm$ and residue field $k$. For a finitely generated $R$-module $M$ we study the {\it Poincar\'e series}
$$
\Po^R_M(z)=\sum_{i\ges 0}\beta_i^R(M)z^i\,,
$$
where $\beta_i^R(M)$ are the {\it Betti numbers} of $M$ defined as $\beta_i^R(M)=\rank_k\Tor_i^R(M,k)$.  We would like to understand when this series is {\it rational}, meaning that it represents a rational function. This problem has a long history, fueled in part by the question attributed to Kaplansky and Serre of whether  $\Po^R_k(z)$ is rational. 


Following Roos \cite{Roos}, we say that $R$ is  {\it good} if all finitely generated $R$-modules have rational Poincar\'e series, sharing a common denominator, and $R$ is {\it bad} otherwise. 

Significant classes of good rings are known. Complete intersections are among them by a result of Gulliksen \cite{Gu74}.   For some of these classes, concrete information on the common denominator has been  successfully used  to further understand the asymptotic behavior of the sequences $\{\beta_i^R(M)\}_{i\ges 0}$  (see \cite{Sun1}, \cite{Sun2}) and to prove results on vanishing of (co)homology (see \cite{ABS}, \cite{Jor}, \cite{Sega1}, \cite{Sega2}).  

 Bad rings exist too, with examples including rings $R$ with  transcendental $\Po^R_k(z)$, cf.\,Anick \cite{Anick}, and rings $R$ admitting  sequences of $R$-modules $\{M_n\}$ with $\Po_{M_n}^R(z)$ rational, but sharing no common denominator, cf.\! Roos \cite{Roos}.

Since both good and bad behavior exist, it is interesting to understand which one is most likely to occur. If we consider families of rings parametrized by points of some geometric object, then this problem can be made precise by asking the question:  Is a generic ring in such a family good or bad? 

The {\it embedding dimension} of $R$ is the integer $e=\rank_k(\fm/\fm^2)$ and, when $R$ is Artinian, the {\it socle degree} of $R$ is the  integer $s$ satisfying  $\fm^s\ne 0=\fm^{s+1}$. Gorenstein Artinian local $k$-algebras of fixed embedding dimension  $e$  and  socle degree $s$ provide a good setting for studying the question above, as they can be parametrized by points in projective space via the Macaulay inverse system. It is an easy consequence of this duality that the length $\lambda(R)$ of a Gorenstein Artinian local $k$-algebra $R$ satisfies 
$$
(*) \qquad \lambda(R)\le \sum_{i=0}^s\varepsilon_i \quad \text{where}\quad \varepsilon_i=\min \left\{ \binom{e-1+s-i}{e-1}, \binom{e-1+i}{e-1}\right\}\,.
$$
When $R$ is a generic algebra in this family,  it is known, cf.\! Iarrobino \cite{Iar84},  that its length is maximal, that is, equality holds in ($*$). 

 In this paper we consider,  more generally,  local Artinian Gorenstein rings $R$ of possibly mixed characteristic. We prove that the inequality ($*$)  holds in this setting as well. We say that $R$ is a  {\it compressed Gorenstein local  ring} of socle degree $s$ and embedding dimension $e$ if equality holds in ($*$).  For such rings one also has $\rank_k(\fm^i/\fm^{i+1})=\varepsilon_i$ for all $i$ with $0\le i\le s$, and the associated graded ring $\agr R$ with respect to the maximal ideal  is Gorenstein. 

Since $R$ is Artinian (thus complete), Cohen's structure theorem gives that $R=Q/I$ for a regular local ring $Q$ and $I\subseteq \fn^2$, where $\fn$ is the maximal ideal of $Q$. We say that $R=Q/I$ is a  {\it minimal Cohen  presentation} of $R$. Note that $\Po_R^Q(z)$ is a polynomial of degree $e$ and it does not depend on the choice of the presentation. 

\begin{Main Theorem}
Let $e$, $s$ be integers such that $2\le s\ne 3$ and $e>1$ and let $(R,\fm,k)$ be a compressed Gorenstein local ring of socle degree $s$  and embedding dimension $e$. Let $R=Q/I$ be any minimal Cohen presentation of $R$. 
  
For every finitely generated $R$-module $M$ there exists a polynomial $p_M(z)\in \mathbb Z[z]$ such that 
$$\Po_M^R(z)=\frac{p_M(z)}{d_R(z)}\,,
$$ where $p_k(z)=(1+z)^e$ and 
$$d_R(z)=1-z\big(\Po^Q_R(z)-1\big)+z^{e+1} (1+z)\,.$$
If  $s$ is even, then $d_R(z)$ depends only on the integers $e$ and $s$ as follows
$$
d_R(z)=1+z^{e+2}+(-z)^{-\frac{s-2}{2}}\cdot \bigg( \big(\sum_{i=0}^s(-1)^i\varepsilon_iz^i\big)\cdot (1+z)^e-1- z^{s+e}\bigg)
$$
and in particular $\Po^R_k(z)=\Po^{\agr R}_k(z)$. 
\end{Main Theorem}

  When $s$ is odd, $\Po^Q_R(z)$, and thus $d_R(z)$, may depend on invariants other than $e$ and $s$. The values of the (graded) Betti numbers of $R$ over $Q$ for generic Gorenstein Artinian graded algebras $R$ of odd socle degree are conjectured by Boij \cite{B}.

In the case $s=2$ all Gorenstein Artinian local rings are compressed and our theorem recovers a  result of  Sj\"odin \cite{Sjo}.  When $s=3$ any ring $R$ for which $\agr R$ is Gorenstein is compressed.  B\o gvad \cite{Bogvad} constructed examples of compressed Gorenstein local rings with $s=3$ and $e=12$ which have transcendental Poincar\'e series, hence rationality results cannot be achieved without stronger hypotheses when $s=3$. Nevertheless, using results of Henriques and \c Sega  \cite{HS}, Conca, Rossi and Valla \cite{CRV},   Elias and Rossi  \cite{ER}, we obtain the following statement:
\begin{Corollary} 
If  $e$ and $s$ are positive integers, then a  generic Gorenstein Artinian algebra of socle degree $s$ and embedding dimension $e$ is good. 
\end{Corollary}

To prove the Main Theorem we construct a Golod homomorphism from a hypersurface ring $P$ onto $R$; a result of Levin then gives the desired conclusions on the rationality of the Poincar\'e series. 

 Throughout, we try to state the intermediary results with only the hypotheses that are needed for their proof.  In Section 1, and more precisely Lemma \ref{Golod-hom}, we give a general result that is used to prove the Golod property. Section 2 collects some change-of-ring homological properties and computations needed for later development.  In Section 3 we build a set-up in which we can verify the hypotheses of Lemma \ref{Golod-hom} and perform computations of Poincar\'e series. 

 Compressed Gorenstein local rings are introduced and discussed in Section 4. Section 5  ties together the previous work and shows that compressed Gorenstein local rings satisfy the hypothesis of the more general Theorem \ref{compute}; this gives the proof of the Main Theorem.  

In the last section we pursue a suggestion of  J\"urgen Herzog and note that $R/\fm^i$ is a Golod ring for $R$ as in the Main Theorem and all $i$ with $2\le i\le s$. When $s$ is even, a simple proof of this fact is given, and it yields a more direct proof of the formula for $d_R(z)$ in the theorem.

\section{Golod homomorphisms}

Throughout the paper, the notation $(R,\fm,k)$ identifies $R$ as a local Noetherian commutative ring with unique maximal ideal $\fm$ and residue field $k=R/\fm$. If $(R,\fm,k)$ is a local ring and $M$ is a finitely generated $R$-module, we denote $\beta_i^R(M)$ its Betti numbers, defined as
$
\beta_i^R(M)=\rank_k\Tor_i^R(M,k)\,.
$
The {\it Poincare series} of $M$ over $R$ is the formal power series
$$
\Po^R_M(z)=\sum_{i\ge 0}\beta_i^R(M)z^i\,.
$$
\
\begin{bfchunk}{Golod homomorphisms.}
\label{Massey}
Let $\varkappa\col (P,\fp,k)\to (R,\fm,k)$ be a surjective homomorphism of local rings. Let $\mathcal D$ be a minimal free resolution of $k$ over $P$ with a graded-commutative  DG algebra structure; such a resolution always exists, as proved independently by Gulliksen \cite{Gu68} and Schoeller \cite{Sch}. We set $\mathcal A=\mathcal D\otimes_PR$. If $x\in\mathcal A$, let $|x|$ denote the degree of $x$ and set $\ov x=(-1)^{|x|+1}x$.

 Let $\bsh =\{h_i\}_{i\ge1} $ be a homogeneous basis of the graded $k$-vector space $\HH_{\ge 1}(\mathcal A)$. 
According to Gulliksen \cite{Gu2}, the homomorphism $\varkappa: P \to R $ is said to be \emph{Golod} (or $ \mathcal A$ admits a {\it{trivial Massey operation}})  if   there is a function
$\mu\col\bigsqcup_{n=1}^{\infty}\bsh^n\to \mathcal A$ satisfying:
\begin{align}
\label{eq:tmo1}
\mu&(h)
\text{ is a cycle in the homology class of $h$
for each } h\in\bsh\,;
   \\
\label{eq:tmo2}
\dd\mu&(h_{1},\dots,h_{n})
=\sum_{i=1}^{n-1}\ov{\mu(h_{1},\dots,h_{i})}
\mu(h_{i+1},\dots,h_{{n}})
\quad\text{for each }n\geq 2\,;
    \\
   \label{eq:tmo3}
\mu&(\bsh^n)\subseteq \fm \mathcal A\quad\text{for each }n\geq 1\,.
\end{align}
\end{bfchunk}

The following result formalizes ideas from \cite{Low} and \cite{AIS}. 

\begin{lemma}
\label{Golod-hom}
With the notation in {\rm \ref{Massey}}, suppose there exists a positive integer $a$ such that
\begin{enumerate}[\quad\rm(1)]
\item The  map $\Tor_i^P(R,k)\to \Tor_i^P(R/\fm^{a},k)$ induced by the projection $R\to R/\fm^a$ is zero for all $i>0$. 
\item The map $\Tor_i^P(\fm^{2a},k)\to \Tor_i^P(\fm^{a},k)$  induced by the inclusion $\fm^{2a}\hookrightarrow \fm^a$ is zero for all $i\ge 0$. 
\end{enumerate}
Then $\varkappa$ is a Golod homomorphism. Furthermore, the Massey operation $\mu$ can be constructed so that $\operatorname{Im}(\mu)\subseteq \fm^{a}\mathcal A$. 
\end{lemma}

\begin{proof}
We construct a trivial Massey operation $\mu$  on $\HH_{\ge 1}(\mathcal A)$. 

The map in (1) can be indentified with the map
$$
\HH_i(\mathcal A)\to \HH_i(\mathcal A/\fm^{a}\mathcal A)
$$
induced by the projection $\mathcal A\to \mathcal A/\fm^{a}\mathcal A$. The map in (2) can be identified with the map
$$
\HH_i(\fm^{2a}\mathcal A)\to \HH_i(\fm^{a}\mathcal A)
$$
induced by the inclusion of complexes  $\fm^{2a}\mathcal A\hookrightarrow \fm^a\mathcal A$. The hypothesis (1) gives that every element in $\HH_{\ge 1}(\mathcal A)$ can be represented as $\cls(x)$ for some $x\in \fm^{a}\mathcal A$. Then if $x_1$ and $x_2$ are such representatives of two elements in  $\HH_{\ge 1}(\mathcal A)$ we have  $x_1x_2\in \fm^{2a}\mathcal A$.  The hypothesis (2) gives that $x_1x_2$ is a boundary in $\fm^a\mathcal A$.

Thus for any two elements $\cls(x_1)$ and $\cls(x_2)$ in  $\HH_{\ge 1}(\mathcal A)$ one can pick representatives $x_1$, $x_2$ in $\fm^{a}\mathcal A$ such that $x_1x_2$ is the boundary of an element in $\fm^{a}\mathcal A$. This property allows for an inductive definition of the Massey operation $\mu$.
\end{proof}

In the case of interest for this paper, condition (1) of Lemma \ref{Golod-hom} will be verified using Lemma \ref{(1)} below.

\begin{bfchunk}{Maps induced by  powers of the maximal ideal.}
\label{powers}
Let $(R,\fm,k)$ be a complete local ring.  Let $R= Q/I$ be a minimal Cohen presentation, with $(Q,\fn,k)$ a regular local ring and $I\subseteq \fn^2$.  We let $\agr R$, respectively $\agr Q$, denote the associated graded rings: 
 $$
\agr R=\bigoplus_{j\ge 0}\fm^j/\fm^{j+1}\quad\text{and}\quad \agr Q=\bigoplus_{j\ge 0}\fn^j/\fn^{j+1}\,.
$$

If $M$ is an $R$-module, we denote 
\begin{equation}
\label{nu}
\nu_i^R(M)\colon \Tor_i^R(\fm M,k)\to \Tor_i^R(M,k)
\end{equation}
the maps induced by the inclusion $\fm M\subseteq M$. \c Sega \cite[Theorem 3.3]{Sega-powers} shows that
$$
\nu_i^R(\fm^j)=0\quad\text{for all $i\ge 0$ and all $j\ge \reg_{\agr Q}(\agr{R})$}\,,
$$
where $\reg_{\agr Q}(-)$ denotes Castelnouvo-Mumford regularity over $\agr Q$, which is a polynomial ring. It follows that the maps
$$
\rho_i^R(\fm^{j})\colon \Tor_i^R(R/\fm^{j+1},k)\to \Tor_i^R(R/\fm^j,k)
$$
induced by the projection $R/\fm^{j+1}\to R/\fm^j$ are zero for all $i>0$ and all $j\ge \reg_{\agr Q}(\agr{R})$.
\end{bfchunk}

\begin{lemma}
\label{(1)}
Let $(Q,\fn,k)$ be a regular local ring and $t\ge 2$. Let $I\subseteq \fn^t$  and set $R=Q/I$ and $\fm=\fn/I$. If $P=Q/(h) $ with $h\in I\smallsetminus \fn^{t+1}$, then the maps
 \begin{align*}
&\rho_i^P(\fm^{t-1})\colon\Tor_i^P(R/\fm^t,k)\to \Tor_i^P(R/\fm^{t-1},k)\\
&\rho_i^Q(\fm^{t-1})\colon\Tor_i^Q(R/\fm^t,k)\to \Tor_i^Q(R/\fm^{t-1},k)
\end{align*}
induced by the projection $R/\fm^t\to R/\fm^{t-1}$ are zero for all $i>0$.
\end{lemma}

\begin{proof}
Set $\fp=\fn/(h)$, the maximal ideal of $P$.  Since $I\subseteq \fn^t$, we have that $R/\fm^t=P/\fp^t$ and $R/\fm^{t-1}=P/\fp^{t-1}$; these identifications are canonical. Thus, to show  $\rho_i^P(\fm^{t-1})=0$ for all $i>0$, it suffices to verify that the map 
 $$
\rho_i^P(\fp^{t-1})\colon\Tor_i^P(P/\fp^t,k)\to \Tor_i^P(P/\fp^{t-1},k)
$$
is zero for all $i>0$. This follows from \ref{powers},  since $\reg_{\agr Q}(\agr P)=t-1$. 

The same argument yields the conclusion for $\rho_i^Q(\fm^{t-1})$. 
\end{proof}

\section{A change of ring and homological computations}

Our goal  is to verify the hypotheses of Lemma \ref{Golod-hom}, under appropriate assumptions on the ring $R$, and with a  ring $P$ as in Lemma \ref{(1)}.  In this section we set up the stage, by analyzing homological properties related to the change of ring given by the projection map $Q\to Q/(h)$, where $Q$ is a local ring and $h$ is a non zero-divisor on $Q$. 

\begin{chunk}
\label{t-1}
Throughout this section,  $(Q,\fn,k)$ and $(P,\fp,k)$ denote  local rings such that $P=Q/(h)$ for some non zero-divisor $h\in \fn^t\smallsetminus \fn^{t+1}$ with $t\ge 2$ and $M$ is a finitely generated $P$-module. For each $i$ we denote $\varphi^M_i$ the  map
\begin{equation}
\label{varphi-def}
\varphi^M_i\colon \Tor_i^Q(M,k)\to \Tor_i^P(M,k)
\end{equation}
induced by the canonical projection $Q\to P$. A free resolution of $M$ over $P$ can be constructed as in \cite[Theorem 3.1.3]{Avr98} from a minimal free resolution of $M$ over $Q$; this construction is due to Shamash \cite{Sha}. If $h\in \fn \ann_Q(M)$, then this resolution is minimal, see  \cite[Remark 3.3.5]{Avr98},  and it follows that $\varphi_i^M$ is injective for all $i$. In particular, if $\fn^{t-1}M=0$ then $\varphi_i^M$ is injective for all $i$.
\end{chunk}

If $\Lambda$ is graded vector space over $k$ such that $\rank_k(\Lambda_i)$ is finite for each $i$ and $\Lambda_i=0$ for $i<r$, then the formal power series 
$$
HS_{\Lambda}(z)=\sum_{i\ges r}\rank_k(\Lambda_i)z^i
$$
is called the {\it Hilbert series} of $\Lambda$.

\begin{chunk}
\label{P-Q}
We discuss now the relationship between  $\Po^Q_M(z)$ and $\Po^P_M(z)$. 
As explained in the proof of \cite[3.2.2]{Avr98}, the construction  of Shamash leads to the definition of maps $\chi_i^M$ which fit into an exact sequence
\begin{align*}
\dots\to \Tor_{i-1}^P(M,k)\to &\Tor_i^Q(M,k)\xrightarrow{\varphi_i^M}\Tor_i^P(M,k)\xrightarrow{\chi_i^M}\Tor_{i-2}^P(M,k)\to\\
& \to \Tor_{i-1}^Q(M,k)\xrightarrow{\varphi_{i-1}^M} \Tor_{i-1}^P(M,k)\to \dots
\end{align*}
A rank count in this sequence gives: 
\begin{equation*}
\beta_i^P(M)-\beta_{i-2}^P(M)=\beta_i^Q(M)-\rank_k(\Ker \varphi_i^{M})-\rank_{k}(\Ker \varphi_{i-1}^{M})\quad \text{for all $i$,}
\end{equation*}
which yields : 
\begin{align*}
\sum_{i\ges 0}\beta_i^P(M)z&^i-z^2\sum_{i\ges 0}\beta_{i-2}^P(M)z^{i-2}=\sum_{i\ges 0} \beta_i^P(M)z^i-\\
&-\sum_{i\ges 0} \rank_k(\Ker \varphi_i^{M})z^i-z\cdot \sum_{i\ges 0}\rank_{k}(\Ker \varphi_{i-1}^{M})z^{i-1}
\end{align*}
and hence
\begin{equation}
\label{P-formula}
{\displaystyle \Po^P_M(z)=\frac{\Po^Q_M(z)-(1+z)\cdot HS_{\Ker \varphi_*^{M}}(z)}{1-z^2}\,.}
\end{equation}
\end{chunk}

Below, we use the notation introduced in \eqref{nu} and \eqref{varphi-def}. 

\begin{lemma}
\label{change}
 If $\fn^tM=0$, then the following statements are equivalent:
\begin{enumerate}[\quad\rm(1)]
\item $\nu_i^P(M)=0$ for all $i$;
\item $\rank_k(\Ima\nu_i^Q(M))=\rank_k(\Ker \varphi_i^{M})$ for all $i$. 
\end{enumerate}
\end{lemma}

\begin{proof}
Consider the exact sequence
$$
0\to \fn M\to M\to M/\fn M\to 0
$$
and the one induced in homology: 
\begin{align*}
\dots \xrightarrow{} \Tor_i^Q(\fn M,k)\xrightarrow{\nu_i^Q(M)}&\Tor_i^Q(M,k)\xrightarrow{\quad} \Tor_i^Q(M/\fn M,k)\xrightarrow{\quad }\\
&\xrightarrow{\quad }\Tor_{i-1}^Q(\fn M,k)\xrightarrow{\nu_{i-1}^Q(M)} \Tor_{i-1}^Q(M,k)\xrightarrow{} \dots
\end{align*}
A rank count (similar to the one spelled out in \ref{P-Q}) in the last sequence gives the formula: 
\begin{equation}
\label{Q}
z\Po^Q_{\fn M}(z)+\Po^Q_{M}(z)-\Po^Q_{M/\fn M}(z)=(1+z)\cdot HS_{\Ima\nu_*^Q(M)}(z)\,.
\end{equation} 
Proceed similarly over the ring $P$ and obtain: 
\begin{equation}
\label{P}
z\Po^P_{\fn M}(z)+\Po^P_{M}(z)-\Po^P_{M/\fn M}(z)=(1+z)\cdot HS_{\Ima\nu_*^P(M)}(z)\,.
\end{equation}
 Since $\fn^{t-1}(\fn M)=\fn^tM=0=\fn^{t-1}(M/\fn M)$, \ref{t-1} gives that $\varphi_i^{\fn M}$ and $\varphi_i^{M/\fn M}$ are  injective for all $i$. Using \eqref{P-formula}, we have thus: 
\begin{equation}
\label{2}
\Po^{P}_{M/\fn M}(z)=\frac{\Po^Q_{M/\fn M}(z)}{1-z^2}\quad\text{and}\quad \Po^{P}_{\fn M}(z)=\frac{\Po^Q_{\fn M}(z)}{1-z^2}\,.
\end{equation}
Combining equations \eqref{Q}-\eqref{2} and \eqref{P-formula}, we obtain: 
$$
HS_{\Ima\nu_*^P(M)}(z)=\frac{ HS_{\Ima\nu_*^Q(M)}(z)-HS_{\Ker \varphi_*^{M}}(z)}{1-z^2}\,.
$$
We have thus $\nu_i^P(M)=0$ for all $i$ if and only if $ HS_{\Ima\nu_*^Q(M)}(z)=HS_{\Ker \varphi_*^{M}}(z)$. 
\end{proof}

\begin{lemma}
\label{injective}
Let  $N$ be a submodule of  $M$ and let $i$ be an integer such that the map
$$\Tor_i^Q(N,k)\to \Tor_i^Q(M,k)$$
induced by the inclusion $N\hookrightarrow M$  is zero. If $\varphi^{M/N}_i$ is injective, then $\varphi_i^M$ is injective.
\end{lemma}
\begin{proof}
Starting with the exact sequence
$$
0\to N\to M\to M/N\to 0
$$
one obtains a commutative diagram: 
\[
\xymatrixrowsep{1.7pc}
\xymatrixcolsep{1.8pc}
\xymatrix{
\Tor_{i}^Q(N,k)\ar@{->}[r]^0&\Tor_i^Q(M,k)\ar@{->}[d]^{\varphi_i^{M}}\ar@{->}[r]&\Tor_i^Q(M/N,k)\ar@{->}[d]^{\varphi_i^{M/N}}\\
&\Tor_i^P(M,k)\ar@{->}[r]&\Tor_i^P(M/N,k)
}
\]
We know that rightmost map in the top row is injective,  since the preceding map in the top sequence is zero. Also, $\varphi_i^{M/N}$ is injective by hypothesis. The rightmost commutative diagram then gives that $\varphi_i^{M}$ is injective. 
\end{proof}

In what follows, we let $K^Q$ denote the Koszul complex on a minimal system  of generators of $\fn$.   If $M$ is a $Q$-module, we set $K^M=K^Q\otimes_QM$. Note that $K^Q$ is a DG algebra and $K^M$ has an induced DG module structure over $K^Q$, in the sense of \cite[\S 1]{Avr98}. We denote by $Z_i(K^M)$ the cycles  of $K^M$ of homological degree $i$. The notation $\partial$ identifies the differential of any of the complexes that we work with. 
 
\begin{proposition}
\label{structure} Assume $(Q,\fn,k)$ is regular and $h\in \fn^2$. Set $e=\dim(Q)$ and $P=Q/(h)$. Let $G\in K^Q_1$ such that $\dd G=h$. If $M$ is a $P$-module such that
$$
Z_e(K^M)=GZ_{e-1}(K^M)
$$
then $\varphi^M_e=0$. 
\end{proposition}

\begin{proof}
Let $g$ denote the image of $G$ in $K^P_1$. Note that $K^Q$ is a minimal free resolution of $k$ over the regular local ring $Q$. Since $P=Q/(h)$ with $h\in \fn^2$, a  minimal free resolution $\mathcal D$  of $k$ over $P$ has  underlying graded algebra 
$$
\mathcal D_*=K^P_*\otimes_P{\mathsf\Gamma}^P_*\big(PY\big)
$$
and differential induced from that of $K^P$, together with the relation $\dd Y=g$,  and extended according to DG$\Gamma$ algebra axioms; the notation ${\mathsf\Gamma}^P_*\big(PY\big)$ stands for the divided power polynomial algebra on the variable $Y$ of degree $2$.  This construction goes back to Cartan; for the description in terms of  DG$\Gamma$ algebra structures we refer to \cite{GL} or \cite{Avr98}. 

Set $D^M=\mathcal D\otimes_PM$. We view $K^P$ as a subcomplex of $\mathcal D$ and $K^M$ as a subcomplex of $\mathcal D^M$. 
 Computing homology over $Q$ by means of $K^Q$ and homology over $P$ by means of $\mathcal D$, one can identify the map $\varphi_e^M\colon \Tor_e^Q(M,k)\to \Tor_e^P(M,k)$ with the canonical  map
$$
\HH_e(K^M)\to \HH_e(\mathcal D^M)
$$
induced by the inclusion  $K^M\subseteq \mathcal D^M$. 
Note that $K^M$ has a canonical structure of DG module over $K^P$ and $\mathcal D^M$ has a canonical structure of DG module over $\mathcal D$. 

We have $\HH_e(K^M)=Z_e(K^M)$. Let $z\in Z_e(K^M)$. By the hypothesis, we have $z=Gz'$ with $z'\in Z_{e-1}(K^M)$. We can view $z'$ as a cycle in $\mathcal D^M$. 
Using the DG module axioms for $\mathcal D^M$, we have: 
$$
z=Gz'=gz'=\dd(Y)z'=\dd(Yz')- Y\dd(z')=\dd(Yz')
$$
Hence $z$ becomes a boundary in $\mathcal D^M$. 
\end{proof}

\section{A  Golod homomorphism}

In this section we identify a general set of hypotheses that allow for an application of Lemma \ref{Golod-hom}, and computations of Poincar\'e series. 
The assumptions are as follows: 

\begin{chunk}
\label{rings}
Let $s,t$ be positive integers with 
\begin{equation}
\label{ineq}
2\le 2t-2\le s\le 3t-4\,.
\end{equation}
We  set $r=s+1-t\,.$
\medskip

\noindent Let $(Q,\fn,k)$  $(P,\fp,k)$, $(R, \fm,k)$  be three local rings such that: 
\begin{enumerate}
\item $Q$ is a regular local ring of dimension $e$. 
\item $R=Q/I$ with $I\subseteq \fn^t$ and $I\not\subseteq \fn^{t+1}$,  and $P=Q/(h)$ with $h\in I \smallsetminus \fn^{t+1}$.
\item $\fm^{s+1}=0\ne \fm^s$. 
\end{enumerate}
Let $\varkappa\col P\to R$ be the canonical projection.

The inequalities in \eqref{ineq} give $2t-2\ge r+1\ge r\ge t-1$, hence we have  inclusions
\begin{equation}
\label{incl}
\fm^{2t-2}\subseteq \fm^{r+1}\subseteq \fm^r\subseteq \fm^{t-1}\,.
\end{equation}

\end{chunk}
Let $M$ be an $R$-module. Note that  $M$ inherits both a $Q$-module and a $P$-module structure. We denote $\Soc(M)$ the {\it socle} of $M$, defined as 
$$\Soc(M)=\{x\in M\mid \fm x=0\}\,.
$$
Since $Q$ is regular of dimension $e$, we compute $\Tor_e^Q(M,k)$ as $\HH_e(K^M)$, hence $\Tor_e^Q(M,k)$ can be canonically identified with $\Soc(M)$.

We  use the notation $\varphi_i^M$  as introduced  in \eqref{varphi-def} and $\nu_i^Q(M)$, $\nu_i^P(M)$ as in \eqref{nu}. 
Lemma  \ref{Golod-hom} finds an application as follows: 

\begin{proposition}
\label{Golod}
If $\nu_i^P(\fm^r)=0$ for all $i\ge 0$, then $\varkappa$ is Golod. 
\end{proposition}

\begin{proof}
We verify the hypotheses of Lemma  \ref{Golod-hom} with $a=t-1$. 

 Condition (1) of Lemma \ref{Golod-hom}  follows immediately  from Lemma \ref{(1)}.  The inclusions in \eqref{incl} show that the map $\Tor_i^P(\fm^{2a},k)\to \Tor_i^P(\fm^{a},k)$  induced by the inclusion $\fm^{2a}\subseteq \fm^a$ factors through $\nu_i^P(\fm^r)$ and it is thus zero for all $i$, hence  condition (2) of Lemma \ref{Golod-hom}  holds as well. 
\end{proof}

\begin{theorem}
\label{compute}
In addition to the assumptions in {\rm \ref{rings}}, assume the following: 
\begin{enumerate}[\quad\rm(a)]
\item $\Soc(R)\subseteq \fm^{r+1}$;
\item  $\varphi_e^{\fm^r}=0$;
\item $\nu_i^Q(\fm^r)=0$ for all $i<e$.
\end{enumerate}
Then $\nu_i^P(\fm^r)=0$ for all $i\ge 0$, and consequently $\varkappa$ is Golod. 

Furthermore,  for every finitely generated $R$--module $M$ there exists a polynomial $p_M(z)\in\mathbb Z[z]$ with
$$
\Po_M^R(z)d_R(z)=p_M(z)
$$
and such that $p_k(z)=(1+z)^e$, where 
$$d_R(z)=1-z(\Po^Q_R(z)-1)+az^{e+1}(1+z)$$ 
with $a=\rank_k\Soc(\fm^r)$. 
\end{theorem}

The proof of the Theorem, mainly a consequence of Proposition \ref{Golod}, is given at the end of the section. 
Some preliminaries are needed in order to establish the formula for $d_R(z)$. 

\begin{lemma}
\label{varphi}
Under the hypotheses of Theorem {\rm \ref{compute}},   the following hold: 
\begin{enumerate}[\quad\rm(1)]
\item $\varphi_i^{\fm^j}$ is injective for all $i<e$ and all $j$ with $t-1\le j\le r$. 
\item $HS_{\Ker \varphi_*^R}(z)=z+az^e$.
\end{enumerate}
\end{lemma}

\begin{proof}
(1)  Let $j$ with $t-1\le j\le r$. The inequalities $t-1+j\ge 2t-2\ge r+1\ge r\ge j$ give a sequence of inclusions: 
$$
\fm^{t-1+j}\subseteq \fm^{r+1}\subseteq \fm^r\subseteq \fm^j
$$
The hypothesis that $\nu_i^Q(\fm^r)=0$  for all $i<e$   shows that  the  map
$$
\Tor_i^Q(\fm^{t-1+j},k)\to \Tor_i^Q(\fm^{j},k)
$$
induced by the inclusion $\fm^{t-1+j}\subseteq \fm^{j}$  is zero for all $i<e$. 

In  Lemma \ref{injective} take $M=\fm^{j}$  and $N=\fm^{t-1+j}$ and note that $\varphi_i^{M/N}$ is injective for all $i$ because $\fn^{t-1}(M/N)=0$, and hence \ref{t-1} applies. Thus Lemma \ref{injective} gives that $\varphi_i^{\fm^{j}}$ is injective for all $i<e$.

(2)  We compute now $\Ker(\varphi_i^R)$ for all $i$. Clearly, we have $\Ker(\varphi_i^R)=0$  for $i>e$ and $\Ker(\varphi_0^R)=0$, since  $\varphi_0^R$ is an isomorphism. 
\medskip

\noindent{\it Claim 1.} $\varphi_e^R=0$.  Indeed, the hypothesis $\Soc(R)\subseteq \fm^{r+1}$ gives  $\Soc(\fm^r)=\Soc(R)$, hence $\Tor_e^Q(\fm^r,k)\cong \Tor_e^Q(R,k)$. We have thus a commutative diagram:
\[
\xymatrixrowsep{1.5pc}
\xymatrixcolsep{1.6pc}
\xymatrix{
\Tor_e^Q(\fm^r,k)\ar@{->}[d]^{\varphi_e^{\fm^r}}\ar@{->}[r]^{\cong}&\Tor_e^Q(R,k)\ar@{->}[d]^{\varphi_e^R}\\
\Tor_e^P(\fm^r,k)\ar@{->}[r]&\Tor_e^P(R,k)
}
\]
which shows that $\varphi_e^R=0$ because $\varphi_e^{\fm^r}=0$. 
\medskip

\noindent{\it Claim 2.}  $\Ker(\varphi_i^R)=0$ for all $i$ with $1<i<e$. 

Set $\ov R=R/\fm^{t-1}$. Recall from  Lemma \ref{(1)} that the maps 
\begin{align*}
&\rho_i^P(\fm^{t-1})\colon \Tor_i^P(R/\fm^t,k)\to \Tor_i^P(\ov R,k)\\
&\rho_i^Q(\fm^{t-1})\colon \Tor_i^Q(R/\fm^t,k)\to \Tor_i^Q(\ov R,k)
\end{align*}
are zero for all $i>0$. Consequently, the maps 
$$
\Tor_i^P(R,k)\to \Tor_i^P(\ov R,k)\quad \text{and}\quad
\Tor_i^Q(R,k)\to \Tor_i^Q(\ov R,k)
$$
induced by the canonical projection $R\to \ov R$ are zero. 

 Note that $\fn^{t-1}\ov R=0$, hence $\varphi_i^{\ov R}$ is injective for all $i$, as discussed in  \ref{t-1}. Consider  the exact sequence
$$
0\to \fm^{t-1}\to R\to \ov R\to 0
$$
and write  the long exact sequences induced by applying $-\otimes_Pk$, respectively $-\otimes_Qk$, and also the change of rings sequences as in \ref{P-Q}. The maps $\chi_i$ introduced in \ref{P-Q}  can also be understood as given by the action of the {\it Eisenbud operator}, in the particular case of a hypersurface, see \cite[Construction 9.1.1]{Avr98}.  Eisenbud operators are natural in the module arguments and they commute with connecting maps induced by short exact sequences, cf.\,\cite[Proposition 9.1.3]{Avr98}. Together with the result of (1) that the maps $\varphi_i^{\fm^{t-1}}$ are injective for all $i<e$, these facts yield for each $i$ with $1<i<e$  the commutative diagram below,  with exact rows and columns.
\[
\xymatrixrowsep{0.2pc}
\xymatrixcolsep{1.4pc}
\xymatrix{
&0\ar@{->}[dd]&0\ar@{->}[dd]&&&\\
&&&&&\\
0\ar@{->}[r]&\Tor_{i+1}^Q(\ov R,k)\ar@{->}[r]\ar@{->}[ddd]^{\varphi_i^{\ov R}}&\Tor_i^Q(\fm^{t-1},k)\ar@{->}[ddd]^{\varphi_i^{\fm^{t-1}}}\ar@{->}[r]&\Tor_i^Q(R,k)\ar@{->}[ddd]^{\varphi_i^R}\ar@{->}[r]&0\\
&&&&&\\
&&&&&\\
0\ar@{->}[r]&\Tor_{i+1}^P(\ov R,k)\ar@{->}[r]\ar@{->}[ddd]^{\chi_{i+1}^{\ov R}}&\Tor_i^P(\fm^{t-1},k)\ar@{->}[r]\ar@{->}[ddd]^{\chi_i^{\fm^{t-1}}}&\Tor_i^P(R,k)\ar@{->}[r]&0\\
&&&&&\\
&&&&&\\
0\ar@{->}[r]&\Tor_{i-1}^P(\ov R,k)\ar@{->}[r]\ar@{->}[dd]&\Tor_{i-2}^P(\fm^{t-1},k)\ar@{->}[dd]&&\\
&&&&&\\
&0&0&&&}
\]
A use of the snake lemma in this diagram shows that $\varphi_i^R$ is injective. 
\medskip

\noindent{\it Claim 3.} $\rank_k\Ker(\varphi_1^R)=1$. Indeed, the map $\varphi_1^R$ can be identified with the canonical projection 
$
I/\fn I\to I/(\fn I,h)
$. Since $I\subseteq \fn^t$ and $h\notin \fn^{t+1}$, we conclude $h\notin \fn I$, hence the kernel has rank $1$. 
\end{proof}

We can now prove the theorem. 

\begin{proof}[Proof of Theorem {\rm \ref{compute}}] The hypothesis that $\Soc(R)\subseteq \fm^{r+1}$ yields that $\Soc(\fm^{r+1})=\Soc(\fm^r)$, hence $\nu_e^Q(\fm^r)$ is an isomorphism.

 Note that $\fn^{t}(\fm^r)=\fm^{s+1}=0$. We verify now condition (2) of Lemma \ref{change} for $M=\fm^r$. Let $i<e$. 
Since $\varphi_i^{\fm^r}$ is injective by Lemma \ref{varphi}(1) and $\nu_i^Q(\fm^r)=0$ by hypothesis, we have $\Ker(\varphi_i^{\fm^r})=0=\Ima\nu_i^Q(\fm^r)$. When $i=e$, we have 
$\Ker(\varphi_e^{\fm^r})=\Soc(\fm^r)=\Ima\nu_e^Q(\fm^r)$. Lemma \ref{change} gives thus $\nu_i^P(\fm^r)=0$ for all $i$, and  Proposition \ref{Golod} then gives  that $\varkappa$ is Golod. 

Thus $R$ is a homomorphic image of a hypersurface ring via a Golod homomorphism. Results of Levin cf.\,\cite[Proposition 5.18]{Avr88} give then the remainder of the conclusion, except for the formula for $d_R(z)$.

Note that
$$
d_R(z)=(1+z)^e\cdot (\Po_k^R(z))^{-1}\,.
$$
We need thus to find a formula for $\Po^R_k(z)$. Since $\varkappa$ is Golod, the trivial  Massey operation can be used to describe the minimal free resolution of $k$ over $R$, see \cite[Prop.~1]{Gu2}, showing that the following  formula holds: 
\begin{equation}
\label{Golod-formula2}
\Po^R_k(z)=\frac{\Po_k^P(z)}{1-z(\Po^P_R(z)-1)}\,.
\end{equation}
Using the conclusion (2)  of  Lemma \ref{varphi} in formula \eqref{P-formula}, we have
\begin{equation}
\label{e1}
\Po^P_R(z)=\frac{\Po^Q_R(z)-(1+z)(z+az^e)}{1-z^2}\,.
\end{equation}
Since $P$ is a hypersurface, we also have:
\begin{equation}
\label{e2}
\Po^P_k(z)=\frac{(1+z)^e}{1-z^2}\,.
\end{equation}
Plugging \eqref{e1} and \eqref{e2} into \eqref{Golod-formula2} and simplifying, we obtain: 
\begin{equation}
\Po^R_k(z)=\frac{(1+z)^e}{1-z(\Po^Q_R(z)-1)+az^{e+1}+az^{e+2}}\,.
\end{equation}
This formula gives the desired expression for $\Po_k^R(z)$, and thus for $d_R(z)$. 
\end{proof}

\section{Compressed Gorenstein local rings}

In this section we introduce the notion of compressed Gorenstein local ring. This notion has been previously studied in the equicharacteristic case. In Proposition \ref{compressed} we prove that some of the properties known for $k$-algebras hold in the general setting as well. We then proceed to verify some of the hypotheses of Theorem \ref{compute}. 

Let $(R,\fm,k)$ be a local ring. The {\it Hilbert function} $h_R$  of $R$ is defined by 
$$
h_R(i)=\rank_k(\fm^i/\fm^{i+1})\quad\text{for all $i\ge 0$}\,.
$$
The number $h_R(1)$ is the  {\it embedding dimension} of $R$, denoted $\edim(R)$.   
Assume from now that $R$ is Artinian. 

\begin{chunk}
\label{v}
 Let $R=Q/I$ be a minimal Cohen presentation, with $(Q,\fn,k)$ a regular local ring with $\dim(Q)=e$. If  $t\ge 1$, then the following conditions are equivalent: 
\begin{enumerate}
\item $I\subseteq \fn^t$;
\item $h_R(t-1)=h_Q(t-1)$;
\item $h_R(i)=h_Q(i)$ for all $i$ with $0\le i\le t-1$. 
\end{enumerate}
Since $h_Q(i)=\displaystyle{\binom{e-i+1}{e-1}}$ for all $i\ge 0$, these conditions are independent of the choice of the minimal Cohen presentation, and so is the number
\begin{equation}
v(R)=\max \{i\colon I\subseteq \fn^i\}\,.
\end{equation}
One has thus $v(R)\ge t$ if and only if $I\subseteq \fn^t$, and  $v(R)= t$ if and only if  $I\subseteq \fn^t$ and  $I\not\subseteq \fn^{t+1}$.

We say that an element $ g\in R $ has {  valuation} $i$ in $R$ and we write $v_R(g)=i$ if $ g\in \fm^i\smallsetminus \fm^{i+1}$. 
If $g$ is an element of valuation $j$ of $R, $ we denote by $g^*$ the homogeneous element of degree $j$  in $\agr R$   given by $\ov g \in \fm^j/\fm^{j+1}.$  Note that if $a$ and $b$ are elements of $R$ with $v_R(a)=i$ and $v_R(b)=j$, then we have 
$$
a^*b^*=(ab)^* \iff a^*b^*\ne 0\iff v_R(ab)=i+j\,.
$$
 
We recall that 
$\agr R=\agr Q/I^*$, where $ I^*$ is the homogeneous ideal generated by all the elements $g^*\in \agr Q$ for $g\in I$. Note that $v(R)=v(\agr R)$. This number can be thought of as the {\it initial degree} of $I^*$, since we have equalities: 
$$
v(R)=\inf\{v_Q(h)\colon h\in I\}=\inf\{\deg(h^*)\colon h^*\in I^*\}=v(\agr R)\,.
$$
\end{chunk}

We let $\displaystyle{\left\lceil x\right\rceil}$ denote  the smallest integer not less than $x$. If $M$ is an $R$-module, we denote by $\lambda(M)$ its length and we set $M^\vee=\Hom_R(M,R)$.

The following result is known when the ring $R$ is equicharacteristic, and it is thus an Artinian  $k$-algebra, see \cite{IE}, \cite{Iar84}. The characteristic assumption is unnecessary, and we provide a complete proof in order to make this point. 

\begin{proposition}
\label{compressed}
Let $(R,\fm,k)$ be a Gorenstein local Artinian ring of socle degree $s$ and embedding dimension $e>1$. We set
$
t=\displaystyle{\left\lceil \frac{s+1}{2}\right\rceil}$ and 
$$
\varepsilon_i=\min\left\{ \binom{e-1+s-i}{e-1}, \binom{e-1+i}{e-1}\right\} \quad\text{for all $i$ with $0\le i\le s$.}
$$
Then 
$
\lambda(R)\le \sum_{i=0}^e \varepsilon_i
$
and the following conditions are equivalent: 
\begin{enumerate}[\quad\rm(1)]
\item $\lambda(R)=\sum_{i=0}^e \varepsilon_i$;
\item $h_R(i)=\varepsilon_i$ for all $i$ with $0\le i\le s$;
\item  $v(R)\ge t$ and $\ann(\fm^t)=\fm^{s+1-t}$.
\end{enumerate}
These conditions imply the following ones:
\begin{enumerate}[\quad\rm(a)]
\item $v(R)=t$;
\item $\ann(\fm^i)=\fm^{s+1-i}$ for all $i$ with $0\le i\le s+1$;
\item $\agr R$ is Gorenstein.
\end{enumerate}
\end{proposition}

\begin{definition}
Let $R$ be as in Proposition \ref{compressed}. We say that $R$ is a {\it compressed Gorenstein ring} of socle degree $s$ and embedding dimension $e$ if $R$ has maximal length, that is, $\lambda(R)=\sum_{i=0}^e \varepsilon_i$. 
\end{definition}

\begin{proof}[Proof of Proposition {\rm \ref{compressed}}]
Set $\varepsilon_i=0$ when $i<0$. 
Let $R=Q/I$ be a minimal Cohen presentation. We have 
$$h_Q(j)=\binom{e-1+j}{e-1}=\varepsilon_j\quad\text{for all}\quad j\le t-1\,.
$$
 Since $h_R(j)\le h_Q(j)$ for all $j\ge 0$ we have thus
\begin{equation}
\label{ineq1}
h_R(j)\le \varepsilon_j\quad \text{for all $j\le t-1$}\,.
\end{equation}
For each  $i$ with $0\le i\le s+1$ we have $\fm^{s+1-i}\subseteq \ann(\fm^i)$, hence 
\begin{equation}
\label{ineq2}
\lambda(R/\fm^i)=\lambda\big((R/\fm^i)^\vee\big)=\lambda(\ann(\fm^i))\ge \lambda(\fm^{s+1-i})\,.
\end{equation}
The first equality above is due to the fact that $R$ is Gorenstein. Next, note that 
$$
\lambda(R/\fm^i)=\sum_{j=0}^{i-1}h_R(j)\quad \text{and}\quad \lambda(\fm^{s+1-i})=\sum_{j=s+1-i}^sh_R(j)=\lambda(R)-\sum_{j=0}^{s-i}h_R(j)\,.
$$
The inequality \eqref{ineq2} becomes thus:
\begin{equation}
\label{ineq4}
\lambda(R)\le \sum_{j=0}^{i-1}h_R(j)+\sum_{j=0}^{s-i}h_R(j)\,.
\end{equation}
When we take $i=t$ in \eqref{ineq4} and we use \eqref{ineq1} we obtain:
\begin{equation}
\label{ineq3}
\lambda(R)\le \sum_{j=0}^{t-1}h_R(j)+\sum_{j=0}^{s-t}h_R(j)\le \sum_{j=0}^{t-1}\varepsilon_j+\sum_{j=0}^{s-t}\varepsilon_j=\sum_{j=0}^{t-1}\varepsilon_j+\sum_{j=t}^{s}\varepsilon_{s-j}=\sum_{j=0}^s \varepsilon_j\,.
\end{equation}
For the second inequality, note that $s-t\le t-1$, hence $h_R(j)\le \varepsilon_j$ for all $j\le s-t$ by \eqref{ineq1}, and for the last equality, note that $\varepsilon_j=\varepsilon_{s-j}$ for all $j$ with $0\le j\le s$.

(1)$\iff$(3):  Equalities hold in \eqref{ineq3} if and only if equalities  hold in \eqref{ineq1} for all $j\le t-1$ and in \eqref{ineq2} for $i=t$. Equalities hold in \eqref{ineq1} for all $j\le t-1$ if and only if  $h_R(i)=h_Q(i)$ for all $i\le t-1$, and this is equivalent to  $v(R)\ge t$ by \ref{v}.  Equalities hold in \eqref{ineq2} for $i=t$ if and only if $\lambda(\ann(\fm^t))=\lambda(\fm^{s+1-t})$, and this is equivalent to  $\ann(\fm^t)=\fm^{s+1-t}$. 

Obviously, (2)$\implies$(1). We now show (3)$\implies$(2). 
Assume (3) holds. Since $v(R)\ge t$, we know that $h_R(i)=\varepsilon_i$ for all $i\le t-1$ and $I\subseteq \fn^t$.   
We prove by induction on $n: $    
\begin{equation} \label{ann}  \ann(\fm^n)=\fm^{s+1-n}  \quad\text{for all $n$ with $t\le n\le s+1$.}
\end{equation}
The basis for the induction is the case $n=t$, which holds by assumption. Assume now that the statement holds for $n\ge t$. 
We prove  that $\ann(\fm^{n+1})=\fm^{s-n}$. Let $x\in R$ with $x\fm^{n+1}=0$. Then $x\fm\subseteq \ann(\fm^n)=\fm^{s+1-n}$. In particular, if $X$ denotes the preimage of $x$ in $Q$, then: 
$$
X\fn\subseteq \fn^{s+1-n}+I\subseteq \fn^{s+1-n}
$$
where the inclusion is due to the fact that $I\subseteq \fn^t\subseteq \fn^{s+1-n}$, since we have inequalities $t \ge s+1-t \ge s+1-n. $  Since $Q$ is regular, it follows that $X\in \fn^{s-n}$ and hence $x\in \fm^{s-n}$. This shows  that $\ann(\fm^{n+1})=\fm^{s-n}$ and it  finishes the induction. 
 
We prove now  $h_R(i)=\varepsilon_i$ for all $t \le i\le s+1$. Note that    
$$
(R/\fm^i)^\vee\cong \ann(\fm^i) \quad\text{and}\quad (R/\fm^{i+1})^\vee\cong \ann(\fm^{i+1})  \,.
$$
hence, using duality, we have   $$h_R(i)=\rank_k(\fm^i /\fm^{i+1}) =\rank_k\big(\!\ann (\fm^{i+1})/ \ann(\fm^i)\big)\,.$$  By  (\ref{ann}), we have  then
$$h_R(i)=\rank_k\big(\!\ann (\fm^{i+1})/ \ann(\fm^i)\big)=\rank_k(\fm^{s-i}/\fm^{s+1-i}) =h_R(s-i) $$
for all $i$ with  $t \le i\le s+1$. Note that $s-i\le 2t-1-i\le t-1$ for all such $i$, hence we further have:
$$
h_R(i)=h_R(s-i)\le \varepsilon_{s-i}=\varepsilon_i\,.
$$
Assume now that (1)--(3) hold. To prove (a), note that we already know $v(R)\ge t$. To show that equality holds, use \ref{v}, noting that 
$$
h_R(t)=\varepsilon_t<h_Q(t)\,.
$$
For part  (b) note that the equalities $h_R(j)=\varepsilon_j$ for all $j$ force equalities in \eqref{ineq4}, and thus in \eqref{ineq2}, for all $i$ with $0\le i\le s$. 

We prove part (c) as  a consequence of (b). To show that $\agr R$ is Gorenstein, it is enough to prove that $\Soc (\agr R)\subseteq ({\agr{\fm}})^s$.   Assume $x^* \in  \Soc (\agr R)$ for some $x \in R $ with $v_R(x)=i \le s-1 $ .  Then $ x \fm \subseteq  \fm^{i+2},$ hence $x \fm^{s-i} \subseteq  \fm^{s+1} = 0$  and thus $x \in \ann(\fm^{s-i}) =\fm^{i+1}$ from (b), a contradiction. 
\end{proof}

We now proceed to verify some of the hypotheses of Theorem \ref{compute} in the case of compressed Gorenstein local rings. More precisely, Lemma \ref{overQ} below proves condition (c) of Theorem \ref{compute} and Proposition \ref{hc} proves condition (b). 

\begin{chunk}
\label{RQ}
For the remainder of the section, we assume that  $(R,\fm,k)$ is a compressed Gorenstein local ring of socle degree $s$ and embedding dimension $e>1$.

Let $R=Q/I$ be a minimal Cohen presentation where $(Q,\fn, k) $ is a regular local ring and $I\subseteq \fn^2$.  We set    $t=v(R)$. By Proposition \ref{compressed}, we  know $t=\displaystyle{\left\lceil (s+1)/2\right\rceil}$. Also, we set $r=s+1-t$. Note that $r=t-1$ when $s=2t-2$ is even and $r=t$ when $s=2t-1$ is odd. 
\end{chunk}

\begin{lemma}
\label{overQ}
The map 
$$
\nu_i^Q(\fm^{r})\colon \Tor_i^Q(\fm^{r+1},k)\to \Tor_i^Q(\fm^{r},k)
$$
is zero for all $i<e$ and is bijective for $i=e$. 
\end{lemma}

\begin{proof}
Since $R$ is Gorenstein, $\Soc(\fm^{r})=\Soc(\fm^{r+1})=\fm^s$,  hence $\nu_e^Q(\fm^{r})$ is bijective. Proposition \ref{compressed} gives $\ann(\fm^t)= \fm^{r}$ and $\ann(\fm^{t-1})=\fm^{r+1}$, hence
$$
(R/\fm^t)^\vee \cong \ann(\fm^t)=\fm^{r}\quad\text{and}\quad
 (R/\fm^{t-1})^\vee\cong \ann(\fm^{t-1})=\fm^{r+1}\,.
$$
Proposition \ref{compressed} and \ref{v} give that $I\subseteq \fn^t$. In particular,  $R/\fm^j=Q/\fn^j$ for all $j\le t$. We have thus canonical isomorphisms
\begin{equation}
\label{vee}
(Q/\fn^t)^\vee \cong \fm^{r}\quad\text{and}\quad 
 (Q/\fn^{t-1})^\vee\cong \fm^{r+1}\,.
\end{equation}
Since $R$ is Gorenstein, note that $M^\vee\cong \Ext^e_Q(M,Q)$ for any  finitely generated $R$-module $M$. Consequently, if 
$$
F:\qquad F_e\to \dots \to F_i\to F_{i-1}\to\dots\to F_0
$$
is a minimal free resolution of $M$ over $Q$, then 
$$
F^*:\qquad (F_0)^*\to \dots (F_{i-1})^*\to( F_{i})^*\to\dots \to (F_e)^*
$$
 is a minimal free resolution of $M^\vee$ over $Q$, where $(F_i)^*=\Hom_Q(F_i,Q)$.  In particular, we have canonical identifications: 
$$
\Tor_i^Q(M^\vee,k)=\HH_{i}(F^*\otimes_Qk)=\HH_{e-i}(F\otimes_Qk)=\Tor_{e-i}^Q(M,k)\,.
$$
In view of \eqref{vee}, we have $\nu_i(\fm^r)=0$ for $i<e$  if and only if 
$$
\rho_{e-i}^Q(\fn^{t-1})\col \Tor_{e-i}^Q(Q/\fn^t,k)\to \Tor_{e-i}^Q(Q/\fn^{t-1},k)
$$
is zero for all $i<e$. 
Since $Q$ is regular, we can use \ref{powers} to prove the last statement. 
\end{proof}

Below, we  adopt the following notational convention: if $X$ is an element of $Q$, then $x$ denotes its image in $R$. 

\begin{lemma} \label{structure2} With the notation in {\rm\ref{RQ}}, assume $k$ is infinite. Let $h \in I$ with $v_Q(h)=t$.  There exists then a minimal system of generators  $X_1,..., X_e$ of $\fn$ such that  
\begin{enumerate}[\quad\rm(1)]
\item  $ (h)= (X_1^t  + C)$ with  $C \in (X_2,...,X_e).$
\item  $\fm^s=x_1^{t-1}\ann_R(x_2,...,x_e).$
\end{enumerate}
\begin{proof} 

 Since $Q$ is regular, $\agr Q$ is a polynomial ring over $k$. Say $\agr Q=k[Z_1, \dots, Z_e]$ where $Z_1, \dots, Z_e$ are variables of degree $1$.  Let   $h\in I$ with $v_Q(h)=t$, hence $\deg(h^*)=t$ in $\agr Q$.   Then after a generic change of coordinates in $\agr Q $ we may assume that    $h^*= Z_1^t + A $ where $A \in (Z_2,\dots, Z_e)_t .$   Now let $X_1, \dots, X_e$ be a minimal system of generators of $\fn$ such that $X_i^*=Z_i. $ We have then
$$
h=X_1^t+L\quad\text{with}\quad L\in (X_2,\dots, X_e)+\fn^{t+1}\,.
$$
Further, we can write $L=C'+X_1^{t}B$ with $C'\in (X_2, \dots, X_e)$ and $B\in (X_1)$, hence 
\begin{equation}
\label{h}
h=X_1^t(1+B)+C'  \in I
\end{equation} 
where $1+B$ is an unit in $Q $  and (1) follows taking $C=(1+B)^{-1} C'. $   In particular we deduce
\begin{equation}
\label{x1}
x_1^t\in (x_2, \dots, x_e)\,.
\end{equation}

Set $\fq=\ann(x_2, \dots, x_n)$ and recall that  $r=s+1-t.$  Now we prove the following claims: 

\noindent{\it Claim 1.} $\fq\subseteq \fm^r$. 

Since $R$ is Gorenstein, the statement is equivalent to $\ann(\fm^r)\subseteq \ann(\fq)$.  Noting that $\ann(\fq)=(x_2, \dots, x_n)$ we need  thus to show that 
$$
\ann(\fm^r)\subseteq (x_2, \dots, x_e)\,.
$$
Note that $\ann(\fm^r)=\fm^t$ by Proposition \ref{compressed}. Using \eqref{x1},  we have
$$\fm^t= (x_2, \dots, x_e)\fm^{t-1}+(x_1^t)\subseteq (x_2, \dots, x_e)\,.$$

\noindent{\it Claim 2.} $\fq\not\subseteq \fm^{r+1}$. 

Since $R$ is Gorenstein, the statement is equivalent to $\ann(\fm^{r+1})\not\subseteq \ann(\fq)$. Since  $\ann(\fm^{r+1})=\fm^{t-1}$ by Proposition \ref{compressed}, we need to show
$$
\fm^{t-1}\not\subseteq (x_2,\dots, x_e)\,.
$$
Assume that the inclusion holds. Then in the ring $Q$ this implies
$$
\fn^{t-1}\subseteq (X_2, \dots, X_e)+I\subseteq (X_2, \dots, X_e)+\fn^t\,.
$$
The second inclusion is due to the fact that $I\subseteq \fn^t$. 
Using Nakayama, it follows that $\fn^{t-1}\subseteq (X_2,\dots, X_e)$. This is a contradiction, since  $\dim Q=e$. 
\medskip

\noindent {\it Claim 3.} $x_1^{t-1}\fq\ne 0$. 

Assume that $x_1^{t-1}\fq =0$. This means that $x_1^{t-1} \in \ann(\fq)=(x_2, \dots, x_e)$ and, in  particular, it implies $\fm^{t-1}\subseteq (x_2,\dots, x_e)$. This cannot happen, see  Claim 2. 
\medskip

Since $\fq\subseteq \fm^r, $   we have $x_1^{t-1}\fq\subseteq \fm^s$. Since $x_1^{t-1}\fq\ne 0$ and $R$ is Gorenstein, it follows that $\fm^s=x_1^{t-1} \fq$.
 \end{proof}
\end{lemma}

\begin{proposition}
\label{hc}
With the notation in {\rm\ref{RQ}}, let $h \in I\smallsetminus \fn^{t+1}  $  and let  $(P,\fp,k)$ be the local ring defined by $P=Q/(h)$. 
Then the map $$\varphi_e^{\fm^r}\colon \Tor_e^Q(\fm^r,k)\to \Tor_e^P(\fm^r,k)$$  is zero. 
\end{proposition}

\begin{proof}
Set $k'=k(y)$, $Q'=Q[y]_{\fn[y]}$, $P'=P[y]_{\fp[y]}$ and                                                                                                                                                                                                                                                                                                                                                        $\fm'=\fm R[y]_{\fm[y]}$. Since the extensions $Q\to Q'$ and $P\to P'$ are faithfully flat, note that $\varphi_e^{\fm^r}=0$  if and only if 
$$\varphi_e^{(\fm')^r}\colon \Tor_e^{Q'}((\fm')^r,k')\to \Tor_e^{P'}((\fm')^r,k')$$  is zero. We may assume thus that $k$ is infinite. 

 We show that the hypotheses of Proposition \ref{structure} are satisfied and then we apply this result. Let $X_1, \dots, X_e$ be a minimal system of generators of $\fn$ as in Lemma \ref{structure2}. We may assume then $h=X_1^{t} + C$ for some $C\in (X_2, \dots, X_e)$  and  
$$
\Soc(\fq)=\Soc(R)=\fm^s=x_1^{t-1} \fq\quad\text{where}\quad \fq = \ann_R(x_2,...,x_e)\,.
$$ 
In Proposition \ref{structure}, take  $K^Q$ to be the Koszul complex on $X_1, \dots, X_e$ and $M=\fq$. 

Let $T_1, \dots, T_e$ be exterior algebra variables such that $K^R=\bw_*^R\big(RT_1\oplus\dots\oplus RT_e\big)$, with $\dd T_i=x_i$. Since $h=X_1^{t} + C$ we can pick $G\in K^Q_1$ with $\dd(G)=h$ such that the image $g$ of $G$ in $K^R$ satisfies $g=x_1^{t-1}T_1+\alpha_2T_2+\dots+\alpha_eT_e$ with $\alpha_2, \dots, \alpha_e\in R$. 

Identify $K^{\fq}$ with $\fq K^R$. We need to verify the equality
$$
Z_e(\fq K^R)=gZ_{e-1}(\fq K^R). 
$$
Indeed, we have: 
\begin{align*}
Z_e(\fq K^R)&=\Soc(\fq)T_1\dots T_e=(x_1^{t-1} \fq) T_1\dots T_e=\fq(x_1^{t-1}T_1)(T_2, \dots, T_e)=\\
&=\fq(g-\alpha_2T_2-\dots \alpha_eT_e)T_2\dots T_e=\fq gT_2\dots T_e=g(\fq T_2\dots T_e)
\end{align*}
If $q\in \fq$, then $q\dd(T_i)=qx_i=0$ for all $i\ne 1$ and the graded Leibnitz rule for DG algebras gives that $\dd(q T_2\dots T_e)=0$. Thus $\fq T_2\dots T_e\subseteq Z_{e-1}(\fq K^R)$. 

 Hence the hypotheses of Proposition \ref{structure} are satisfied, and we conclude $\varphi_e^{\fq}=0$. 

Recall that  $\fq\subseteq \fm^r$ by Claim 1 of the proof of  Lemma \ref{structure2}. An argument simliar to the one in the proof of Lemma \ref{varphi}(2) shows that  $\varphi_e^{\fq}=0$ implies $\varphi_e^{\fm^r}=0$. 
\end{proof}

\section{The main results}

We are now ready to state our main result, which is mainly a consequence of Theorem \ref{compute}, in view of the results of Section 4. 

\begin{theorem}
\label{thm1}
Let $e$, $s$ be integers such that $2\le s\ne 3$ and $e>1$ and let $(R,\fm,k)$ be a compressed Gorenstein local ring of socle degree $s$  and embedding dimension $e$. Let $R=Q/I$ be any minimal Cohen presentation of $R$. Let $h\in I$ with $v_Q(h)=v(R)$ and let $(P,\fp,k)$ be the local hypersurface ring defined by $P=Q/(h)$. 

The canonical map $\varkappa\col P\to R$ is then Golod. 
Furthermore, for  every finitely generated $R$--module $M$ there exists a polynomial $p_M(z)\in\mathbb Z[z]$ with
$$
\Po_M^R(z)d_R(z)=p_M(z)
$$
and such that $p_k(z)=(1+z)^e$, where 
\begin{equation}
\label{t1}
d_R(z)=1-z(\Po^Q_R(z)-1)+z^{e+1}(1+z)\,.
\end{equation}
If  $s$ is even, then $d_R(z)$ depends only on the integers $e$ and $s$ as follows
\begin{equation}
\label{t2}
d_R(z)=1+z^{e+2}+(-z)^{-\frac{s-2}{2}}\cdot \bigg(HS_{\agr R}(-z)\cdot (1+z)^e-1- z^{s+e} \bigg)
\end{equation}
and in particular $\Po^R_k(z)=\Po^{\agr R}_k(z)$. 
\end{theorem}

The  proof is given at the end of the section. 

\begin{remark}
The Poincar\'e series of Gorenstein local rings of embedding dimension $e\le 4$ have been classified in \cite{Avr87}. According to this classification, the compressed Gorenstein local  rings with $e=4$  are of type GGO. This is because the type GGO is the only one in the list  for which $d_R(z)$  has degree $6$. 
\end{remark}

\begin{remark}
\label{sjodin}
When $s=2$, the formula \eqref{t2}  becomes
 $$d_R(z)= HS_{\agr R}(-z)\cdot (1+z)^e\,.$$ 
 The result of the Theorem is subsumed, in this case,  by the results of Sj\"odin \cite{Sjo}, who proves the formulas for Poincar\'e series,  and  Avramov et.\! al.\! \cite{AIS}, who prove the conclusion about the Golod homomorphism. 
\end{remark}
\begin{remark}
Our methods cannot recover the results of \cite{HS} for $s=3$. Furthermore, B\o gvad's examples \cite{Bogvad}  show that not all compressed Gorenstein local rings with $s=3$ have rational Poincar\'e series. 
\end{remark}

\begin{bfchunk}{Generic Gorenstein Artinian algebras.}
\label{generic}
Let $D$ denote the divided powers algebra over $k$ on $e$ variables of degree $1$ and let  $Q=k[[x_1,\dots,x_e]]$  be the power series ring in $e$ variables. The ring $Q$ acts then on $D$ by ``differentiation", as described for example in \cite{Iar84}. F. S. Macaulay  proved that there exists a one-to-one correspondence between ideals $I \subseteq Q=k[[x_1,\dots,x_e]] $ such that $R=Q/I $  is an Artinian Gorenstein local algebra and cyclic $Q$-submodules of $D$,  see    \cite[Lemma 1.2]{Iar84}. Thus Gorenstein  Artinian  $k$-algebras of embedding dimension $e$ and socle degree $s$   can be parametrized   by points in a certain projective space.
There exists then a nonempty open (and thus dense) subset of this projective space such that the corresponding algebras have maximal length/Hilbert function, see \cite[Theorem I]{Iar84}. In other words, a {\it generic} Gorenstein Artinian  $k$-algebra of embedding dimension $e$ and socle degree $s$  is compressed. 
\end{bfchunk}

\begin{corollary}
Let $e$ and $s$ be any positive integers. 
Let $R$ be a generic Gorenstein Artinian $k$-algebra of embedding dimension $e$ and socle degree $s$. Then for every finitely generated $R$-module there exists a polynomial $p_M(z)\in\mathbb Z[z]$ with
$$
\Po_M^R(z)d_R(z)=p_M(z)
$$
and such that $p_k(z)=(1+z)^e$. 
\end{corollary}

\begin{proof} 
If $e\le 2$, then $R$ is a complete intersection, and the statement  is due to Gulliksen, see \cite[Theorem 4.1]{Gu74}; in this case, $d_R(z)=(1-z^2)^e$. Assume now $e>3$, and in particular $s\ge 2$.
 By \ref{generic}, $R$ is compressed.  If $s\neq 3, $ the result follows from Theorem \ref{thm1}.  If $s=3$, then we may assume $R$ is graded by a result by Elias and Rossi \cite[Theorem 3.3]{ER}.  Conca, Rossi and Valla \cite[Claim 6.5]{CRV}  prove that a  generic Gorenstein graded algebra $R$ has an element $x$ of degree $1$ such that $\ann(x)$ is principal, and work of Henriques and \c Sega  \cite[Corollary 4.5]{HS} shows  that such algebras  satisfy the desired  conclusion.  
\end{proof}

\begin{remark}
When $s$ is odd, the compressed hypothesis in the theorem is not enough to grant that  $\Po^Q_R(z)$, and thus $d_R(z)$, depends only on $e$ and $s$.  However, for graded rings, formulas for the Betti numbers of $R$ over $Q$ depending only on $e$ and $s$ are conjectured when $R$ is generic,  see \cite{B}.  Thus a formula for $d_R(z)$ is available in the cases when the conjecture is known to hold; see \cite{MM} for some cases. 
\end{remark}

We now prepare to give a  proof of  the theorem. 

The Poincar\'e series of a finitely generated graded module over a standard graded algebra can be defined the same way as in the local case. If $N=\oplus N_i$ is a graded $\agr R$-module and $j$ is an integer, then the notation $N(-j)$ stands for the graded module whose $i$th graded component is $N_{i+j}$.

\begin{lemma}
\label{graded}
Under the hypotheses of Theorem {\rm \ref{thm1}}, if $s$  is even, then 
$$\Po^Q_R(z)=1+z^e+(-z)^{-\frac{s}{2}}\cdot \left(  HS_{\agr R}(-z)\cdot (1+z)^e-1- z^{s+e} \right)\,.
$$
\end{lemma}
\begin{proof}
 Since $s$ is even, Proposition \ref{compressed} gives $s=2t-2$ with $t=v(R)$.  By Proposition \ref{compressed}(c), $\agr R$ is Gorenstein as well. Thus $\agr R$ is a compressed Gorenstein graded algebra of socle degree $2t-2$ and embedding dimension $e$. 

Set $A=\agr Q$ and $\beta_i=\beta_i^{A}(\agr R)$.  According to \cite[4.7]{Iar84}, a graded minimal free resolution of $\agr R$ over $A$ is as follows: 
\begin{equation} \label{res}
0\to A(-e-s)\to A^{\beta_{e-1}}(-e-t+2)\to\dots\to A^{\beta_1}(-t)\to A\to \agr R\to 0
\end{equation} 
where the $i$th term of the free resolution is $A^{\beta_i}(-i-t+1)$ for $0<i<e$. 

A Hilbert series computation in (\ref{res}) gives the following: 
$$
 HS_{\agr R}(z)= HS_A(z)(1-\beta_1z^t+\beta_2z^{t+1}+\dots +(-1)^{e-1}\beta_{e-1}z^{t+e-2}+(-1)^e z^{s+e})\,.\\
$$
We further obtain
\begin{align*}
\Po_{\agr R}^A(-z)&=1-\beta_1z+\beta_2z^2+\dots+(-1)^e\beta_ez^e\\
&=z^{-t+1}\left (1-\beta_1z^t+\beta_2z^{t+1}+\dots +(-1)^{e-1}\beta_{e-1}z^{t+e-2}+(-1)^e z^{s+e}\right)+\\
    &\qquad\qquad +1+(-z)^e-z^{-t+1}\cdot (1+(-1)^ez^{s+e})\\
&=z^{-t+1}\cdot HS_{\agr R}(z)\cdot (1-z)^e+1+(-z)^e-z^{-t+1}\cdot (1+(-1)^ez^{s+e})\\
&=1+(-z)^e+z^{-t+1}\cdot \left(  HS_{\agr R}(z)\cdot (1-z)^e-1- (-1)^ez^{s+e} \right)
\end{align*}
where the second equality uses the fact that $\beta_e=1$, and the third equality uses the equation above and the fact that $ HS_A(z)=(1-z)^{-e}$.
We have thus  
$$
\Po_{\agr R}^A(z)=1+z^e+(-z)^{-t+1}\cdot \left(   HS_{\agr R}(-z)\cdot (1+z)^e-1- z^{s+e} \right)\,.
$$

Finally, the fact that the graded resolution of $\agr R$ over $A$ is pure  gives $\Po_{\agr R}^{A}(z)=\Po_R^Q(z)$, see Fr\"oberg \cite[Theorem 1]{Fro}.
\end{proof}

\begin{proof}[Proof of Theorem {\rm\ref{thm1}}]
Set $t=v(R)$. Recall that $t=\displaystyle{\left\lceil (s+1)/2\right\rceil}$ by Proposition \ref{compressed}. We want to apply Theorem \ref{compute}.   First, we need to check that the inequalities of \eqref{ineq} are satisfied. Since  $s=2t-1$ or $s=2t-2$, these  inequalities are consequences of the hypothesis $2\le s\ne 3$. Next, we check the hypotheses (a)--(c): Part (a) follows from the fact that $R$ is Gorenstein. 
Part (b) follows from Proposition \ref{hc}. Part (c) follows from Lemma \ref{overQ}. 

All conclusions of the theorem follow then from Theorem \ref{compute}, except for formula \eqref{t2}. When $s$ is even, \eqref{t2} is obtained by plugging the formula of Lemma \ref{graded} into \eqref{t1}. The formulas in Proposition \ref{compressed}(2) show that $HS_{\agr R}(z)$ depends only on $e$ and $s$. To see that $\Po^R_k(z)=\Po^{\agr R}_k(z)$, note that $\agr R$ satisfies the same hypotheses as $R$, as already  noted in the proof of Lemma \ref{graded}. 
\end{proof}

\begin{remark}
As noted earlier, the statement of our main theorem cannot be extended to the case $s=3$. One may wonder how the condition $s\ne 3$ factors into the proof. This can be traced back to Theorem \ref{compute}, and more precisely the inequalities \eqref{ineq}, which do not hold when $s=3$. Tracing further back, the inequalties \eqref{ineq} were dictated by the attempt to apply Lemma \ref{Golod-hom}. In short, when $s=3$, the ring $R$ is too ``small" to allow for the construction of a trivial Massey operation along the lines of the proof of Lemma \ref{Golod-hom}. 
\end{remark}

\section{Factoring out the socle}

The development described below comes from a suggestion of  J\"urgen Herzog regarding  a different way of computing a formula for $\Po_k^R(z)$  in Theorem \ref{thm1}.

Let $(R,\fm,k)$ be a local ring and $\widehat R=Q/I$ a minimal Cohen presentation, where $\widehat R$ denotes the completion of $R$ with respect to $\fm$. The ring $R$ is said to be {\it Golod} if the induced map $Q\to \widehat R$ is Golod. This definition is independent of the choice of the presentation; see more details in \cite{Avr98}. It is known that $R$ is Golod if and only if the formula: 
$$
\Po^{R}_k(z)=\frac{(1+z)^e}{1-z(\Po^Q_{R}(z)-1)}
$$
holds, where $e$ is the embedding dimension of $R$.  
\medskip

\begin{lemma}
\label{m^s}
Let $R$ be a Gorenstein Artinian  local  ring of embedding dimension $e$. 
The following equality then holds:
$$
\Po^Q_{R/\Soc(R)}(z)=\Po^Q_R(z)+z(1+z)^e-z^e(1+z)\,.
$$
\end{lemma}

\begin{proof}
Note that the map
$$
\nu_i\colon \Tor_i^Q(\Soc(R),k)\to \Tor_i^Q(R,k)
$$
induced by the inclusion $\Soc(R)\subseteq R$ is zero for all $i<e$.   This follows by an argument similar to that in the proof of Lemma \ref{overQ}, using duality; we skip the details.  Also, note that $\nu_e$ is bijective because $\Soc(\Soc(R))=\Soc(R)$. 

Set $S=R/\Soc(R)$. Consider now the exact sequence
$$
0\to \Soc(R)\to R\to S\to 0
$$
and the long exact sequence induced in homology
$$
\dots\to \Tor_i^Q(R,k)\to \Tor_i^Q(S,k)\to \Tor_{i-1}^Q(\Soc(R),k)\to \dots
$$
Note that $\Soc(R)\cong k$. One obtains equalities: 
$$
\beta_i^Q(S)=\beta_i^Q(R)+\beta_{i-1}^Q(k)
$$
for all $i<e$. We also have: 
$$
\beta_e^Q(S)=\beta_{e-1}^Q(k)\quad\text{and}\quad \beta_e^Q(R)=\beta_e^Q(k)=1
$$
and the desired formula follows from here, since $\Po^Q_k(z)=(1+z)^e$. 
\end{proof}

\begin{proposition}
\label{Golod-ring}
Let $R$ be a Gorenstein local ring of embedding dimension $e>1$. The following statements are equivalent: 
\begin{enumerate}[\quad\rm(1)]
\item ${\displaystyle \Po^R_k(z)=\frac{(1+z)^e}{1-z(\Po^Q_R(z)-1)+z^{e+1}(z+1)}}$.
\item $R/\Soc(R)$ is Golod. 
\end{enumerate}
\end{proposition}

\begin{proof}
Set $S=R/\Soc(R)$. By \cite[Theorem 2]{Avr78}, we know that the canonical homomorphism $R\to S$ is Golod and the following formula holds: 
\begin{equation*}
\Po_k^{S}(z)=\frac{\Po_k^R(z)}{1-z^2\Po_k^R(z)}\,.
\end{equation*}
Rearranging this formula, we have:
\begin{equation}
\label{AL}
\Po_k^R(z)=\frac{\Po^{S}_k(z)}{1+z^2\Po^{S}_k(z)}\,.
\end{equation}
Assume that $S$ is Golod, hence 
$$
\Po^{S}_k(z)=\frac{(1+z)^e}{1-z(\Po^Q_{S}(z)-1)}=\frac{(1+z)^e}{1-z(\Po^Q_R(z)-1)-z^2(1+z)^e+z^{e+1}(1+z)}
$$
where we used Lemma \ref{m^s} in the second equality. 
Plugging this into \eqref{AL} and simplifying, one obtains the formula in (1). 

Proceed similarly for the converse. 
\end{proof}

In view of Theorem \ref{thm1}, it follows that $R/\Soc(R)$ is Golod whenever  $R$ satisfies the hypotheses of the theorem. A more straightforward explanation can be given when $R$ has even socle degree, see the proof below.

\begin{proposition}
If $R$ is a Gorenstein compressed local ring of socle degree $s$ with $2\le s\ne 3$, then 
 $R/\fm^i$ is a Golod ring for all $i$ with $2\le i\le s$. 
\end{proposition}
 
\begin{proof}
 When $e=1$, the result follows,  for example,  from \cite[Proposition 6.10]{Sega-powers}. Assume $e>1$. 
 Let $t=v(R)$, so that $s=2t-1$ or $s=2t-2$. 

Assume  $t\le i\le 2t-2$. Set $J=I+\fn^i$. Then $R/\fm^i=Q/J$ and 
$$\fn^{2t-2}\subseteq \fn^{i}\subseteq J\subseteq \fn^t\,.
$$
It is a known fact that $Q/J$ is Golod when $\fn^{2t-2}\subseteq J\subseteq \fn^t$; this was first noted by L\"ofwall \cite{Low2}. For a quick proof of this result, apply Lemma \ref{Golod-hom} to the map $Q\to Q/J$,  with  $a=t-1$. The hypothesis (1) of Lemma \ref{Golod-hom} is verified in this case in view of  Lemma \ref{(1)} and the hypothesis (2) is trivially satisfied, since $\fn^{2t-2}/J=0$. 

When $i\le t$, then $R/\fm^i=Q/\fn^i$ and the ring $Q/\fn^i$ is Golod for all $2\le i$, since $Q$ is regular. 

The only case not covered yet  is when $i=s=2t-1$. Theorem \ref{thm1} and Proposition \ref{Golod-ring} give the conclusion in this case. 
\end{proof}

\begin{remark}
In the even socle degree case, when $R$ is a  standard graded $k$-algebra, Fr\"oberg \cite{Fro2} used a similar approach to determine a formula for $\Po_k^R(z)$. He also gives a formula for the bigraded Poincar\'e series,  see \cite[Theorem 2]{Fro2}. 
\end{remark}

\section*{acknowledgment}
Irena Peeva's series of talks at an MSRI workshop in September 2012 mentioned the problem of deciding rationality of Poincar\'e series for generic rings and served as a catalyst for our collaboration. 

We would like to thank several of our colleagues for entertaining discussions that helped us in the process of simplifying our proofs and obtaining stronger results:  Luchezar Avramov, Aldo Conca, J\"urgen Herzog, Juan Migliore, Kristian Ranestad. We thank  Lars Christensen, Pedro Macias Marques and Oana Veliche for sending us a list of typos, and the referee for suggesting improvements in exposition.  We also acknowledge the use of the computer algebra system Macaulay2, that guided us in the initial stages of this work.


\begin{thebibliography}{99}
\bibitem{Anick}
D.~J.~Anick, 
\textit{A counterexample to a conjecture of Serre}, Ann. of Math. {\bf  115} (1982); 1--33.


\bibitem{AIS}
    L.~L.~Avramov, S.~Iyengar, L.~M.~\c Sega, 
		\textit{Free resolutions over short local rings},
		J. London Math. Soc. {\bf 78} (2008); 459--476.
 
\bibitem{ABS}
L.~L.~Avramov, R.-O.~Buchweitz, L.~M.~\c Sega, \textit{Extensions of a dualizing complex by its ring: commutative versions of a conjecture of Tachikawa}, J.~Pure Appl.~Algebra {\bf 201} (2005); 218--239.

\bibitem{Avr98}
L.~L.~Avramov, \textit{Infinite free resolutions},  
Six lectures on commutative algebra (Bellaterra, 1996), 
Progress in Math. {\bf 166}, Birkh\"auser, Basel, 1998;  1--118.

\bibitem{Avr88} L.~L.~Avramov, A.~R~Kustin, M.~Miller, \textit{Poincar\'e series of modules over local rings of small embedding codepth or small linking number}, J. Algebra {\bf 118} (1988); 162--204. 


\bibitem{Avr87}
L.~L.~Avramov, \textit{ Homological asymptotics of modules over local rings},  Commutative algebra (Berkeley, CA, 1987),  Math. Sci. Res. Inst. Publ. {\bf 15}, Springer, New York, 1989; 33--63. 

\bibitem{Avr78}
L.~L.~Avramov, G.~Levin, \textit{Factoring Out the Socle of a Gorenstein Ring}, J. Algebra {\bf 55} (1978); 74--83. 


\bibitem{Bogvad}
		R.~B\o gvad,
		\textit{Gorenstein rings with transcendental Poincar\'e series},
		Math. Scand. {\bf 53} (1983); 5--15.


\bibitem{B}
M.~Boij, \textit{Betti numbers of compressed level algebras}, J. Pure Appl. Algebra {\bf 134} (1999); 111--131. 



\bibitem{CRV} A. Conca, M.E. Rossi, G. Valla, \textit{ Gr\"obner flags and Gorenstein Artin rings}, Compositio Math.
{\bf 129} (2001); 95--121.

\bibitem{ER}
J.~Elias,  M.E. Rossi, \emph{Isomorphism classes of short {G}orenstein local
  rings via {M}acaulay's inverse system}, Trans. Amer. Math. Soc. \textbf{364}
  (2012); 4589--4604.

\bibitem{Fro2}
 R.~Fr\"oberg, 
\textit{A study of graded extremal rings and of monomial rings}, Math. Scand. {\bf 51} (1982); 22-34. 

 
\bibitem{Fro}
 R.~Fr\"oberg, 
\textit{Connections between a local ring and its associated graded ring}, J. Algebra {\bf 111} (1987); 300--305. 


\bibitem{Gu2}
T.~H.~Gulliksen, \emph{Massey operations and Poincar\'e series of certain local rings}, J. Algebra {\bf 22} (1970); 223--232. 

\bibitem{Gu68} 
T.~H.~Gulliksen, \textit{A proof of the existence of minimal algebra resolutions}, Acta Math. {\bf 120} (1968); 53--58. 

\bibitem{Gu74}
		T.~H.~Gulliksen, 
		\textit{A change of rings theorem with applications to Poincar\'e series and intersection multiplicity}, 
		Math. Scand. {\bf 34} (1974); 167--183. 


 
\bibitem{GL}
T.~H.~Gulliksen, G.~Levin, 
\textit{Homology of local rings}, 
Queen's Papers Pure Appl. Math. {\bf 20}, 
Queen's Univ., Kingston, ON (1969).



\bibitem{HS} I. B. Henriques and L. M. \c Sega, \textit{Free resolutions over short Gorenstein local rings}, Math. Z. {\bf 267} (2011); 645--663. 

\bibitem{IE}
A.~Iarrobino, J.~Emsalem, \textit{Some zero dimensional generic singularities; finite algebras having small tangent space}, Compositio Math. {\bf 36}  (1978); 145--188. 


\bibitem{Iar84}
A.~Iarrobino, 
\textit{Compressed algebras: Artin algebras having given socle degrees and maximal length}, Trans. Amer. Math. Soc. {\bf 285} (1984); 337--378. 

\bibitem{JKM}
		C.~Jacobsson, A.~R.~Kustin, M.~Miller,
		\textit{The Poincar\'e series of a codimension four Gorenstein ring is rational}, 
		J. Pure Appl. Algebra {\bf 38} (1985); 255--275. 

\bibitem{Jor}
D.~A.~Jorgensen, \textit{A generalization of the Auslander-Buchsbaum formula}, J.~Pure Appl.~Algebra {\bf 144} (1999); 145--155. 

\bibitem{MM}
J.~Migliore, R.~M.~Mir\'o-Roig, \textit{On the Minimal Free Resolution of $n+1$ general forms}, Trans. Amer. Math. Soc. {\bf 355} (2002); 1--36. 

\bibitem{Low2} \textit{On the subalgebra generated by the one-dimensional elements in the Yoneda Ext-algebra}, Algebra, algebraic topology and their interactions (Stockholm, 1983), 291--338, Lecture Notes in Math.{\bf 1183}, Springer, Berlin, 1986.

\bibitem{Low}
C.~L\"ofwall, \textit{On the homotopy Lie algebra of a local ring}, J.~Pure Appl.~Algebra {\bf 38} (1985); 305--312. 

\bibitem{Roos}
J.-E.~Roos, \textit{Good and bad Koszul algebras and their Hochschild homology}, J.~Pure Appl.~Algebra {\bf 201} (2005); 295--327.

\bibitem{Sch}
C.~Schoeller, \textit{Homologie des anneaux locaux noeth\'eriens}, C. R. Acad. Sci. PAris \S\'er. A {\bf 265} (1967); 768--771. 

\bibitem{Sega1}L.~M.~\c Sega,
		\textit{Self-tests for freeness over commutative Artinian rings}, J.~Pure~Appl.~Algebra {\bf 215} (2011);1263--1269.

\bibitem{Sega2}
L.~M.~\c Sega,
		\textit{Vanishing of cohomology over Gorenstein rings of small codimension}, Proc.~Amer.~Math.~Soc. {\bf 131} (2003); 2313--2323.

\bibitem{Sega-powers}
L.~M.~\c Sega,
		\textit{Homological properties of powers of the maximal ideal of a local ring}, J. Algebra {\bf 241} (2001); 827--858.


\bibitem{Sha}
J.~Shamash, \textit{The Poincar\'e series of a local ring}, J.~Algebra {\bf 12} (1969); 453--470. 

\bibitem{Sjo}		
		G.~Sj\"odin,
		\textit{The Poincar\'e series of modules over Gorenstein rings with} $\fm^3=0$, 
		Mathematiska Institutionen, Stockholms Universitet, 
		Preprint {\bf 2} (1979).

 \bibitem{Sun1}
L.-C.~Sun, \textit{Growth of Betti numbers of modules over generalized Golod rings}, J.~Algebra {\bf 199} (1998); 88--93. 


 \bibitem{Sun2}
L.-C.~Sun, \textit{Growth of Betti numbers of modules over local rings of small embedding codimension or small linkage number}, J.~Pure~Appl.~Algebra {\bf 96} (1994); 57--71.



\end{thebibliography}
  \end{document}